\documentclass[reqno,11pt]{amsart}

\usepackage[top=1.9cm,bottom=1.9cm,left=2.4cm,right=2.4cm]{geometry}

\usepackage{amsmath, amsthm, amssymb}
\usepackage{mathrsfs, dsfont, amsfonts}
\usepackage{functan}
\usepackage{extarrows}
\usepackage{mathtools}
\usepackage{appendix}

\usepackage[colorlinks,
            linkcolor=purple,
            anchorcolor=purple,
            citecolor=purple
            ]{hyperref}

\usepackage{xcolor}
\usepackage{latexsym}
\usepackage{enumerate}
\usepackage{float}
\usepackage{relsize}
\usepackage{marginnote}
\usepackage{stmaryrd}
\usepackage{esint}
\usepackage{graphicx}
\usepackage{bm}
\usepackage{verbatim} 

\usepackage{cite}
\allowdisplaybreaks
\usepackage{aliascnt}

\theoremstyle{plain}
\newtheorem{theorem}{Theorem}[section]

\newaliascnt{corollary}{theorem}
\newtheorem{corollary}[corollary]{Corollary}
\aliascntresetthe{corollary}

\newaliascnt{lemma}{theorem}
\newtheorem{lemma}[lemma]{Lemma}
\aliascntresetthe{lemma}

\newaliascnt{proposition}{theorem}

\aliascntresetthe{proposition}

\newaliascnt{remark}{theorem}
\newtheorem{remark}[remark]{Remark}
\aliascntresetthe{remark}

\numberwithin{equation}{section}

\def\R{\mathbb{R}} 
\def\T{\mathbb{T}} 
\def\d{\mathrm{d}} 

\def\eps{\varepsilon}

\usepackage{bbm}
\def\ch{\mathbbm{1}}
\def\supp{\operatorname*{supp}} 

\def\ls{\lesssim}

\def\p{\partial} 
\def\dst{\operatorname*{dist}} 
\def\diam{\operatorname*{diam}} 

\title[Vorticity Gradient Growth in 2D Euler]
{Remarks on Linear Growth of Vorticity Gradients and Support Diameters for 2D Euler Flow in Half-Plane}

\author[S.-Q. Chen]{Shaoqing Chen}
\address{Department of Mathematics, Nanjing University, Nanjing
 210093, P. R. China}
\email{sqchen@smail.nju.edu.cn}
\author[Y.-Z. Sun]{Yongzhong Sun}
\address{Department of Mathematics, Nanjing University, Nanjing
 210093, P. R. China}
\email{sunyongzhong@163.com}

\begin{document}

\begin{abstract}
It has been conjectured that generic smooth solutions of the two-dimensional Euler equation exhibit linear growth of vorticity gradients.
We prove an elementary arbitrary-background perturbation principle in the odd symmetric setting. More precisely, for any compactly supported nonnegative function in the half-plane, one can find an arbitrarily small smooth nonnegative perturbation whose associated solution undergoes linear-in-time filamentation in the quadrant.
The main ingredients are the lower bound of the center of mass given by Iftimie-Sideris-Gamblin \cite{Iftimie1999}, and the velocity estimate for the sparse part to capture those slowly moving particles.
\end{abstract}



\maketitle
\tableofcontents

\section{Introduction}

\subsection{Planar Euler Flow with Symmetries}
The planar Euler system can be formulated in terms of the vorticity $\omega=\nabla^\perp v$\footnote{In this paper, we denote $x^\perp = (-x_2, x_1)$ for a planar vector $x=(x_1,x_2)$.} as: 
\begin{equation}\label{euler}\tag{E}
\left\{
\begin{aligned}
\partial_t \omega+v\cdot \nabla \omega = 0 &\text{ in } \R^2,\\
v= \nabla^\perp (-\Delta)^{-1} \omega &\text{ in } \R^2,\\
\omega|_{t=0} =\omega_0. &
\end{aligned}
\right.
\end{equation}
Here, $v$ is the velocity field, and $\omega$ is the vorticity. 
In this full plane case, the Biot-Savart law asserts that the velocity field $v$ can be recovered from the vorticity $\omega$ via convolution with the Biot-Savart kernel, i.e., $$
v(t,x) = \int_{\R^2} K_0(x-y)\omega(t,y)\,\mathrm{d}y,
$$
where
$K_0(x)=\frac{1}{2\pi}\frac{x^\perp}{|x|^2}$.

The well-posedness theory of the two-dimensional Euler system has been extensively studied since the classical work of Wolibner \cite{Wolibner1933}.
In his important work \cite{Yudovich1963}, Yudovich studied the weak solutions of \eqref{euler} in the space $L^1\cap L^\infty$, and established the global existence and uniqueness of such solutions. 
See also classical textbooks \cite{Majda2002,Bahouri2011} for modern treatments on this topic.

A special kind of solution to \eqref{euler} enjoys odd symmetry, that is, $$
\omega_0(x_1,x_2)=-\omega_0(x_1,-x_2).
$$
From the uniqueness of Yudovich solutions, $\omega(t,x)$ preserves the odd symmetry for all time. Denote $$
\mathbb{R}_+^2:=\{x=(x_1, x_2)\in \R^2: x_2 > 0\},
$$
then any odd solution can be considered as the solution in the first quadrant $\mathbb{R}_+^2$ with slip boundary condition, i.e., $v\cdot n = 0$ on $\p \mathbb{R}_+^2$.
The corresponding Biot-Savart law then becomes 
$$
v(t,x)=\int_{\mathbb{R}_+^2} K(x,y)\omega(t,y)\,\mathrm{d}y,
$$
where the kernel is obtained by the method of images as
\begin{equation}
\label{kernel_half_plane}
K(x,y)=\frac{1}{2\pi}\left(\frac{(x-y)^\perp}{|x-y|^2}-\frac{(x-\bar{y})^\perp}{|x-\bar{y}|^2}\right).
\end{equation}

\subsection{Long Time Behavior of the Support, and Main Result}

The transport nature ensures that the vorticity is conserved along the particle trajectories, which implies that the $L^p(\mathbb{R}_+^2)$ norms of the vorticity are conserved for all $p\in[1,\infty]$. In particular, the total vorticity $$
m := \|\omega(t,\cdot)\|_{L^1(\mathbb{R}_+^2)}=\int_{\mathbb{R}_+^2} \omega(t,x)\,\mathrm{d}x 
$$
and the maximum vorticity $$
M := \|\omega(t,\cdot)\|_{L^\infty(\mathbb{R}_+^2)}
$$
are conserved along the flow.

Another important conserved quantity is the (pseudo) energy, which reads 
$$
\begin{aligned}
E[\omega]&:=-\frac{1}{2\pi}\int_{\mathbb{R}_+^2}\int_{\mathbb{R}_+^2}\log\frac{|x-y|}{|x-\bar{y}|}\omega(x)\omega(y)\,\mathrm{d}x\mathrm{d}y \\
&= \frac{1}{4\pi}\int_{\mathbb{R}_+^2}\int_{\mathbb{R}_+^2}\log\left(1+\frac{4 x_2 y_2}{|x-y|^2}\right)\omega(x)\omega(y)\,\mathrm{d}x\mathrm{d}y.
\end{aligned}
$$

We denote the center of mass in $\mathbb{R}_+^2$ as $$
P(t):=\frac{1}{m}\int_{\mathbb{R}_+^2} x \omega(t,x)\,\mathrm{d}x = (P_1(t), P_2(t)).
$$
In the odd setting, the center of mass is not conserved. In \cite[Section 2.2.1]{Iftimie2007}, the following time-independent estimate on the center of mass is established:
\begin{equation}\label{center_of_mass_half_plane}
P_1'(t) \ge V_P, P_2'(t) = 0,
\end{equation}
which implies that there is a particle in the support of $\omega(t)$ with $x_1$-coordinate at least $P_1(0) + V_P t$. Here, a constant lower bound $V_P$ can be explicitly given by
\begin{equation}\label{VP_half_plane}
V_P: = 10^{-3} \frac{E^{4}}{P_2(0)^3 m^{2} M}.    
\end{equation}
Although the details of the proof of \eqref{center_of_mass_half_plane} won't be presented in this paper, a similar estimate for the generalized surface quasi-geostrophic system in the half-plane will be given in \autoref{lem_drift_gsqg}.

Now we state our main results. The mechanism behind the linear growth is the presence of a sparse portion of vorticity whose average horizontal velocity is slower than the drift of the total center of mass, as observed in \cite{Guo2026} for 3D axisymmetric vortex ring setting.

\begin{theorem}\label{thm_diameter_growth_quadrant}
Assume that the initial vorticity $\omega_0$ is non-negative, 
has bounded support, and is not identically zero. 
If $\omega_0$ admits a decomposition $\omega_0 = \omega_{0,d} + \omega_{0,r}$, 
where $\omega_{0,d}$ and $\omega_{0,r}$ are non-negative, 
compactly supported in $\mathbb{R}_+^2$, and $\omega_{0,d}$ satisfies the following sparse condition: 
For some constant $0 \le C_0 < V_P$, $m_d := \|\omega_{0,d}\|_{L^1} > 0$, and
$$
\left\|\omega_{0,d} \right\|_{L^\infty} \le \left(\frac{C_0}{2m}\right)^{2} \left\|\omega_{0,d} \right\|_{L^1}, 
$$
then the diameter of $D(t):= \supp(\omega(t,\cdot)) \cap \mathbb{R}_+^2$ grows at least linearly:
$$
\diam(D(t)) \ge (V_P - C_0) t -O(1).
$$
\end{theorem}

The following corollary is a direct consequence of the above theorem, which addresses the instability of \eqref{euler} in the odd symmetric setting. In various settings, the vorticity gradient is conjectured to grow generically; see \cite[Section 3]{Drivas2023} for a comprehensive review on this topic.

\begin{theorem}\label{thm_gredient_growth}
For each non-zero boundedly supported non-negative $\omega_0 \in L^1\cap L^\infty (\mathbb{R}_+^2)$ and $\eps > 0$, there exists a nonnegative $\tilde{\omega}_0 \in C_c^\infty (\mathbb{R}_+^2)$ such that $\|\tilde{\omega}_0 - \omega_0\|_{L^1} < \eps$, and the solution $\tilde{\omega}(t)$ to the 2D Euler system \eqref{euler} with initial data $\tilde{\omega}_0$ satisfies that, the following three quantities
$$
\liminf_{t\to\infty} \frac{\diam(\supp(\tilde{\omega}(t)))}{t} , 
\quad \liminf_{t\to\infty} \frac{\|\tilde{\omega}(t)\|_{C^\alpha}}{t^\alpha}\ (0<\alpha<1) , 
\quad \liminf_{t\to\infty} \frac{\|\nabla \tilde{\omega}(t)\|_{L^\infty}}{t} 
$$
are all positive.
Moreover, if $\omega_0$ is continuous, we can further require that $\|\tilde{\omega}_0 - \omega_0\|_{L^\infty} < \eps$.
\end{theorem}

Here, we do not require $\omega_0$ to be supported away from the boundary, which makes the result apply to some classical examples of travelling wave solutions, such as the Lamb-Chaplygin dipole, as background vorticity, and reproduces existing results like those in \cite{Choi2022b}.
The statement of \autoref{thm_gredient_growth} can actually be improved to gradient growth in the $L^1$ sense by the co-area formula and the linear diameter growth of the support. Since the construction requires additional details, we state it as a separate theorem.

\begin{theorem}\label{thm_gredient_integral_growth}
For every $\eps > 0$, each non-zero boundedly supported non-negative $\omega_0 \in L^1\cap L^\infty (\mathbb{R}_+^2)$ admits a nonnegative perturbation $\tilde{\omega}_0 \in C_c^\infty (\mathbb{R}_+^2)$ such that $\|\tilde{\omega}_0 - \omega_0\|_{L^1} < \eps$, and the solution $\tilde{\omega}(t)$ to the 2D Euler system \eqref{euler} with initial data $\tilde{\omega}_0$ satisfies that
$$
\|\nabla \tilde{\omega}(t)\|_{L^1} \ge Ct - O(1)
$$
for some constant $C>0$ and all $t\ge 0$.
\end{theorem}

\subsection{Related Works} Now we briefly review the previous work on the long time behavior of vorticity for 2D Euler system, which is a classical topic in fluid dynamics and has been extensively studied in the literature. We discuss some topics that are closely related to our work.

A question naturally following the well-posedness is how the vorticity distribution evolves in time. 
Denote the support of $\omega(t)$ as $D(t)$. 
We are interested in the growth of the diameter of $D(t)$ as $t\to\infty$. 
It is well known that in the full plane case, 
Marchioro-Pulvirenti \cite{Marchioro1994} first established the upper bound of the diameter growth as $\diam(D(t)) = O(t^{1 / 3})$, 
which was later independently improved to $O(\left(t\log(1+t)\right)^{1 / 4})$ by Iftimie-Sideris-Gamblin \cite{Iftimie1999} 
and to $O(t^{1 / 4}\log\log\cdots\log t)$ by Serfati \cite{Serfati1998}. 

The presence of boundary effect may significantly change the behavior of vorticity. In \cite{Iftimie1999}, Iftimie-Sideris-Gamblin established a linear lower bound of the diameter growth for the odd-odd symmetric solutions, which is optimal in view of the velocity upper bound, see \autoref{lem_velocity_estimate}. Here the diameter is understood as the diameter of odd-odd symmetric support in the full plane. See also \cite{Iftimie2003} for the confinement result. Also in the half-plane setting, Choi-Jeong gave an example of linear diameter and gradient growth in \cite{Choi2022b}, which is obtained by perturbing the Lamb dipole, and apply the corresponding stability estimate. The first example of a confinement without symmetry was given by Zbarsky in \cite{Zbarsky2021}, by analyzing global-in-time dynamics of vorticity near a self-similar point vortices.

Another important motivation for the confinement problem is the filamentation of vortex rings in 3D axisymmetric Euler system. Very recently, Guo-Jeong-Zhao in \cite{Guo2026} proved universal linear-in-time filamentation for 3D axisymmetric vortex rings. 
They demonstrated that if a portion of the initial vorticity is distributed beyond a critical geometric distance from the vortex core, 
the weak local induction fails to match the rapid translation speed of the core, inevitably causing the trailing vorticity to stretch linearly along the symmetry axis. This mechanism is a main ingredient in our proof.

The growth rate of vorticity gradient is another important question. After the seminal work of Kiselev-Šverák \cite{Kiselev2014} on the double exponential growth of vorticity gradient for 2D Euler system in the disk, there have been plenty of works on this topic, see for example \cite{Denisov2015, Kiselev2016} for the double exponential growth of vorticity gradient in different bounded domains. Zlatoš, in \cite{Zlatos2025}, established the double exponential growth of vorticity gradient for 2D Euler system in the half-plane (odd-odd symmetric setting actually) by a delicate construction. Also a recent work on superlinear growth is given in \cite{Jeong2025a}.
Very recently, a breakthrough work \cite{Alazard2026} by Alazard-Said established the $C^k$ norm growth of vorticity for generic initial data in $\T^2$ and $\R^2$ with zero mean settings, which states that the initial data with infinite $C^k$ norm growth forms a dense $G_\delta$ set in the corresponding function space.

\subsection{Notations}

Now we declare the notations we will use in the following context. 
We denote by $\R_+^2:=\{x=(x_1, x_2)\in \R^2: x_2 > 0\}$ the half-plane, and for any point $y=(y_1, y_2)$, we define its reflections $\bar y := (y_1, -y_2)$.
For the evolving vorticity $\omega(t,\cdot)$, its support in the half-plane is denoted by $D(t) := \supp(\omega(t,\cdot))\cap \R_+^2$, and we measure its spatial extent using its diameter $\diam(D(t))$. 
We use the notation $\supp \omega_0 \Subset \R_+^2$ to indicate that the support is a compact subset strictly contained in the open half-plane $\R_+^2$.
We denote standard H\"older semi-norms by $[f]_{C^\alpha}$, and the particle flow map of the velocity field $v$ by $\Phi_t$. 

Finally, throughout this section, $C$ denotes a constant which may change from line to line, but is independent of $\omega_0$ and $t$. We also denote $X\ls Y$ or $X=O(Y)$ as $|X|\le CY$. Both notations may contain more dependencies denoted in the subscript. Notice that $O(1)$ usually denotes a quantity whose absolute value is bounded by a constant depending solely on the initial profile $\omega_0$.

\section{Preliminaries}

The following lemma is a basic estimate for Biot-Savart kernel, which shows that the velocity field is bounded for $L^1 \cap L^\infty$ vorticity. See for example \cite[Lemma 2.1]{Iftimie1999} for a proof.

\begin{lemma}\label{lem_velocity_estimate}
For $f\in L^1 \cap L^\infty(\R^2)$, we have
$$
\left|\int_{\R^2} \frac{f(y)}{|x-y|}\,\mathrm{d}y\right| \le (8\pi)^{1 / 2} \|f\|_{L^1}^{1 / 2}\|f\|_{L^\infty}^{1 / 2}.
$$
\end{lemma}

In the following context, we denote 
$$
V_m := (8\pi)^{1 / 2} m^{1 / 2} M^{1 / 2}
$$
to be the upper bound of the velocity field given by \autoref{lem_velocity_estimate}: $|v(t,x)|\le V_m$ for all $t\ge 0$ and $x\in \R_+^2$ (in fact, the velocity field is bounded by $V_m$ in the whole plane $\R^2$ since we can extend $\omega$ to be an odd-odd function in $\R^2$).

In the analysis of regularity of vorticity, we will need the following geometric lemma.

\begin{lemma}\label{lem_geometric}
Assume that for a $C^\alpha$ function $f$ ($0<\alpha\le1$, where $\alpha=1$ stands for Lipschitz continuity), there exists $a\in\R$ such that the level set $E_a:=\{x: f(x) > a\}$ is $C^1$ and connected, and $E_b:=\{x: f(x) > b\}$ has finite Lebesgue measure for some $b<a$.
Then 
$$
[f]_{C^\alpha} \ge (a-b)\left(2\frac{\operatorname*{diam}(E_a)}{|E_b|}\right)^\alpha.
$$
\end{lemma}
\begin{proof}
Note that the H\"{o}lder continuity suggests that $E_b$ contains a $\delta-$neighborhood of $E_a$ with $\delta = \left(\frac{a-b}{[f]_{C^\alpha}}\right)^{1 / \alpha}$. Hence, $$
|E_b| \ge |\{x: \dst(x, E_a) \le \delta\}|.
$$
To evaluate the right hand side, we find, for any $\eps>0$, a direction $e$ such that the projection of $\pi_{e}(E_a)$ in the direction $e$ has length at least $\operatorname*{diam}(E_a) - \eps$. Since $E_a$ is simply connected, the projection of $E_a$ in the direction $e$ is an interval, hence for each point $s\in [0,\operatorname*{diam}(E_a) - \eps]$ in the projection, a point $x_s \in E_a$ is projected to it, and hence the segment $\{x_s + te^\perp: t\in[-\delta, \delta]\}$ is contained in $\{x: \dst(x, E_a) \le \delta\}$. 
This implies by Fubini Theorem that $$
|E_b| \ge |\{x_s + te^\perp: s\in [0,\operatorname*{diam}(E_a) - \eps], t\in[-\delta, \delta]\}| = 2\delta (\operatorname*{diam}(E_a) - \eps).
$$
We conclude the proof by letting $\eps\to 0$ and rearranging the above inequality.
\end{proof}

\section{Growth of the Support and Small-Scale Creation}

We begin with the proof of \autoref{thm_diameter_growth_quadrant}.

\begin{proof}[Proof of \autoref{thm_diameter_growth_quadrant}]
First we establish a key observation: the horizontal center of mass for the sparse component moves slower than $V_P$. Recall that the Biot-Savart kernel satisfies
$$
| K(x,y) | \le \frac{1}{2\pi} \left(\frac{1}{|x-y|} + \frac{1}{|x-\bar{y}|}\right) \le \frac{1}{\pi} \frac{1}{|x-y|}.
$$
We denote $M_d = \|\omega_{0,d}\|_{L^\infty}$. The horizontal center of mass for the sparse part is defined as $P_{1,d}(t) = \frac{1}{m_d} \int_{\R_+^2} x_1 \omega_d(t,x) \mathrm{d}x$. Using the transport nature of $\omega_d(t)$ and the fact that $v$ is divergence free, we have
$$
\begin{aligned}
P_{1,d}'(t) &= \frac{1}{m_d} \int_{\R_+^2} x_1 \p_t \omega_d(t,x)\,\mathrm{d}x 
= -\frac{1}{m_d} \int_{\R_+^2} x_1 \nabla \cdot \left[ v(t,x) \omega_d(t,x)\right]\,\mathrm{d}x \\
&= \frac{1}{m_d}\int_{\R_+^2} v_1(t,x) \omega_d(t,x)\,\mathrm{d}x 
= \frac{1}{m_d}\int_{\R_+^2} \int_{\R_+^2} K_1(x,y) \omega(t,y)\omega_d(t,x)\,\mathrm{d}y\,\mathrm{d}x \\
&\le \frac{2}{\pi m_d}\int_{\R_+^2} \int_{\R_+^2} \frac{1}{|x-y|} \omega(t,y)\omega_d(t,x)\,\mathrm{d}y\,\mathrm{d}x \\
&\le 2\left(\frac{2}{\pi}\right)^{1 / 2} m\left(\frac{M_d}{m_d}\right)^{1 / 2} \le C_0,
\end{aligned}
$$
where we apply \autoref{lem_velocity_estimate} and the sparse condition. 

The inequality $P_{1,d}'(t) \le C_0$ implies the existence of a particle in the support of $\omega_d(t)$ with its $x_1$-coordinate at most $P_{1,d}(0) + C_0 t \le C_0 t + O(1)$. On the other hand, $P_1'(t) \ge V_P$ guarantees there is a particle in the support of $\omega(t)$ with its $x_1$-coordinate at least $P_1(0) + V_P t \ge V_P t - O(1)$. The difference yields the lower bound for the diameter:
$$
\diam(D(t)) \ge (V_P - C_0) t - O(1).
$$
This concludes the proof of \autoref{thm_diameter_growth_quadrant}.
\end{proof} 

Now we turn to constructing smooth perturbations to prove \autoref{thm_gredient_growth}. 

\begin{proof}[Proof of \autoref{thm_gredient_growth}]
Assume first $D_0:= \supp(\omega_0) \Subset Q$. To secure sufficient space for the sparse component, we choose a smooth simply connected domain $U \Subset Q$ such that $D_0 \Subset U$, with its area bounded below by
$$
|U| \ge |D_0| + 2^9 \left( \frac{m}{V_P} \right)^2.
$$
Take the mollification scale $\rho > 0$ small enough such that $\dst(D_0, U^c) > 10\rho$, and the shrunk region $U_\rho^\circ := \{x\in U: \dst(x, U^c) \ge \rho\}$ along with the expanded core $D_\rho := \{x\in Q: \dst(x, D_0) \le \rho\}$ leaves a gap $W := U_\rho^\circ \setminus D_\rho$ with area
$$
|W| \ge 2^8 \left(\frac{m}{V_P}\right)^2.
$$
We define $\omega_1(x) = \ch_U(x) (\omega_0(x)+a)$ for a small amplitude $a>0$, and set the modified initial profile $\tilde{\omega}_0 = \omega_1 * \eta_\rho$, where $\eta_\rho$ is a standard smooth non-negative mollifier supported in $B_\rho:=B(0,\rho)$. 

Since $D_0 \Subset U$ and $\dst(D_0, U^c) > 10\rho$, by properties of convolution, if $x \notin U_\rho^\circ$, then $B(x,\rho) \cap \supp(\omega_0) = \varnothing$ and the fraction of the ball $B(x,\rho)$ intersecting $U$ is strictly less than $1$ (for $x \in \p U_\rho^\circ$) or zero (for $x$ further away). Thus $\tilde{\omega}_0(x) < a$. Meanwhile, for $x \in W = U_\rho^\circ \setminus D_\rho$, the ball $B(x,\rho)$ is entirely contained in $U \setminus D_0$, on which $\omega_1 \equiv a$. Hence $\tilde{\omega}_0(x) = a$ on $W$.
Under the stated positivity assumptions, this precise construction yields
$$
\{\tilde\omega_0 \ge a\} = U_\rho^\circ.
$$

We can ensure $\|\tilde{\omega}_0 - \omega_0\|_{L^1} < \eps$ provided $a$ and $\rho$ are small enough.
As $a, \rho \to 0$, we have the total mass $\tilde m \to m$ and $V_P[\tilde{\omega}_0] \to V_P[\omega_0] := V_P$. This is guaranteed by convergence of not only the $L^1$ norm, center of mass, but also the energy (since our perturbation is supported in a compact set and the kernel is locally integrable).

Fixing $\rho, a$ suitably small guarantees that 
$$
\frac{\tilde{m}}{m}, \frac{V_P[\tilde{\omega}_0]}{V_P} \in [0.99, 1.01].
$$
For the chosen $W$ on which $\tilde{\omega}_0 \equiv a$, we designate the sparse component $\tilde\omega_{0,d} := \tilde\omega_0 \ch_W$. Its amplitude is $\|\tilde\omega_{0,d}\|_{L^\infty} = a$, and its mass is $\|\tilde\omega_{0,d}\|_{L^1} = a|W|$. Thus,
$$
\|\tilde\omega_{0,d}\|_{L^\infty} \le \left(\frac{V_P[\tilde{\omega}_0]/4}{4\tilde m}\right)^2 \|\tilde\omega_{0,d}\|_{L^1},
$$
where we used the ratio bounds to absorb the constants. This precisely satisfies the sparse condition with $C_0 = V_P[\tilde{\omega}_0] / 4$. 
Applying the argument of \autoref{thm_diameter_growth_quadrant} to $\tilde{\omega}$, the support diameter grows at least linearly: 
$$
\diam(\supp(\tilde{\omega}(t))) \ge (V_P[\tilde{\omega}_0] - C_0) t - O(1) \ge \frac{3}{4} (0.99 V_P) t - O(1).
$$

The fast particle phenomenon is formulated for large time. By definition $$P_1[\tilde{\omega}](t) = \frac{1}{\tilde m} \int_Q x_1 \tilde{\omega}(t,x)\,\mathrm{d}x \ge V_P[\tilde{\omega}_0] t \ge 0.99 V_P t$$ for all $t$. If we assume that, for a sequence $t_j \to \infty$, all particles in $\{ x_1 \ge \frac{3}{4} P_1[\tilde{\omega}](t_j) \}$ have vorticity strictly less than $a$, then we can bound the horizontal mass by splitting the integral:
$$
\begin{aligned}
\tilde{m} P_1[\tilde{\omega}](t_j) &= \int_{\{x_1 < \frac{3}{4}P_1[\tilde{\omega}](t_j)\}} x_1 \tilde{\omega}(t_j,x)\,\mathrm{d}x + \int_{\{x_1 \ge \frac{3}{4}P_1[\tilde{\omega}](t_j)\}} x_1 \tilde{\omega}(t_j,x)\,\mathrm{d}x \\
&\le \frac{3}{4} P_1[\tilde{\omega}](t_j) \tilde{m} + \int_{\{ \tilde{\omega}(t_j,x) < a \}} x_1 \tilde{\omega}(t_j,x)\,\mathrm{d}x \\
&\le \frac{3}{4}\tilde{m} P_1[\tilde{\omega}](t_j) + a |U_\rho^\circ| (R_0 + V_{m} t_j),
\end{aligned}
$$
where $R_0 = \sup_{x\in U} x_1$ and $V_m$ is the uniform velocity bound. Since the area of the support $|U_\rho^\circ|$ is conserved and $P_1[\tilde{\omega}](t_j)$ grows at least linearly as $0.99 V_P t_j$, this yields a contradiction for large $t_j \ge T_0$ if we choose the amplitude $a$ small enough (independently of $t$) such that $a |U_\rho^\circ| V_m < \frac{1}{4}\tilde{m} (0.99 V_P)$.
Thus, for all $t\ge T_0$, there is a particle $x_f$ such that $\tilde{\omega}(t, x_f) \ge a$ and $x_{f,1} \ge \frac34 P_1[\tilde{\omega}](t) \ge \frac{3}{4}(0.99 V_P) t$.
On the other hand, the slow particle $x_s$ with $\tilde{\omega}(t,x_s)=a$ and $x_{s,1} \le C_0 t + O(1) \le \frac{1}{4}(1.01 V_P) t + O(1)$ is guaranteed by the sparse component velocity estimate.
This ensures $x_f$ and $x_s$ diverge linearly.

With the two particles $x_s$ and $x_f$, we obtain a large diameter of the $a$-level set $E_a(t) = \{\tilde{\omega}(t)\ge a\}$. Since $\tilde{\omega}(t)$ is transported by a divergence-free flow, the measure of the lower level set $E_{a/2}(t) = \{\tilde{\omega}(t) > a/2\}$ is conserved and finite. An immediate application of \autoref{lem_geometric} provides the lower bound for $\|\tilde{\omega}(t)\|_{C^\alpha}$ and $\|\nabla \tilde{\omega}(t)\|_{L^\infty}$ which scales at least linearly in $t$.

Now we discuss the case when $\supp\omega_0$ touches the boundary. This is not a big problem since we can first perturb $\omega_0$ to be supported away from the boundary like $\omega_0 \ch_{\{x_2 > 10 \rho\}}$, and then apply the above construction. This won't change the smallness of the perturbation in various senses, and the key quantities $m,M,P_2(0),E,V_P$ are only slightly changed, which will not affect the linear growth mechanism.
\end{proof}

The above construction can be further refined to yield the $L^1$-gradient growth by using the co-area formula and the linear growth of the diameter of the support. The key point is to ensure that all level sets $E_\lambda(t) = \{\tilde{\omega}(t) > \lambda\}$ for $0<\lambda<a$ are simply connected, so that their perimeter can be bounded below by their diameter. This is guaranteed by the following lemma, which allows us to construct a smooth function with simply connected superlevel sets, whose proof we postpone.

\begin{lemma}\label{lem_level_set}
For any smooth simply connected domain $U \Subset Q$, there exists a smooth compactly supported function $\psi\in C_c^\infty(Q)$ such that $0 \le \psi \le 1$, $\psi\equiv 1$ on $U$, and for every $0<c<1$, the superlevel set $\{\psi > c\}$ is simply connected.
\end{lemma}

\begin{proof}[Proof of \autoref{thm_gredient_integral_growth}]
For \autoref{thm_gredient_integral_growth}, we need all level sets $E_\lambda(t)$ for $0<\lambda<a$ to be simply connected.
By \autoref{lem_level_set}, we choose a smooth compactly supported function $\psi\in C_c^\infty$ satisfying $\psi\equiv 1$ on $U$ with all $\{ \psi > c \}$ simply connected. We construct
$$
\tilde{\omega}_0 = \omega_0 * \eta_\rho + a\psi.
$$
For $0 < \lambda < a$, noting that $\omega_0 * \eta_\rho$ is supported in $U$ and $\psi=1$ on $U$, the strict superlevel sets $\{ \tilde{\omega}_0 > \lambda \} = U \cup \{\psi > \lambda / a\}$ are homotopically equivalent to $\{\psi > \lambda / a\}$ and thus simply connected.

For evaluating the $L^1$ norm of the gradient, we use the co-area formula for smooth compactly supported functions:
$$
\int_{Q}|\nabla \tilde{\omega}(t,x)|\,\mathrm{d}x = \int_0^\infty \operatorname{Per}(E_\lambda) \,\mathrm{d}\lambda \ge \int_0^a \operatorname{Per}(E_\lambda) \,\mathrm{d}\lambda.
$$
For almost every regular value $\lambda$, we have $\operatorname{Per}(E_\lambda) = \mathcal{H}^1(\partial E_\lambda)$. For connected bounded sets, $\operatorname{Per}(E_\lambda) \ge 2\diam(E_\lambda)$. Thus the integral is at least $2\int_0^a \diam(E_\lambda(t)) \,\mathrm{d}\lambda \ge 2 a \diam(E_a(t))$, which ensures the $L^1$-gradient growth.
\end{proof}

\begin{proof}[Proof of \autoref{lem_level_set}]
By the tubular neighborhood theorem, there exists $\delta>0$ such that the outward normal map
$$
\Phi:\partial U\times(-\delta,\delta)\to \mathbb{R}^2,
\qquad
\Phi(p,t)=p+t\nu(p),
$$
is a diffeomorphism onto a neighborhood of $\partial U$. Here $\nu$ denotes the outward unit normal to $\partial U$.
Choose $0<\varepsilon<\delta$ sufficiently small so that
$\overline{U_\varepsilon}\subset V$,
where
$
U_\varepsilon
=
U\cup \Phi\bigl(\partial U\times[0,\varepsilon)\bigr)
$.
Let $\phi$ be the signed normal coordinate in the tubular neighborhood, so that $\phi(\Phi(p,t))=t$.
Thus $\phi<0$ on the $U$-side of $\partial U$, $\phi=0$ on $\partial U$, and $\phi>0$ on the exterior side.

Choose $\chi\in C^\infty(\mathbb{R};[0,1])$ such that
$$
\chi=1 \quad \text{on } (-\infty,0],
\qquad
\chi=0 \quad \text{on } [\varepsilon,\infty),
\qquad
\chi'<0 \quad \text{on } (0,\varepsilon).
$$
Define $\psi=\chi\circ \phi$ in the tubular neighborhood, extend it by $1$ on $U$, and by $0$ outside $U_\varepsilon$. Since $\chi$ is constant near the gluing regions, this gives a smooth function $\psi\in C_c^\infty(V)$. Clearly $0\le \psi\le 1$ and $\psi\equiv 1$ on $U$.

Now fix $0<c<1$. Since $\chi$ is strictly decreasing on $(0,\varepsilon)$, there is a unique $t_c\in(0,\varepsilon)$ such that
$\chi(t_c)=c.$
Hence
$$
\{\psi>c\}
=
U\cup \Phi\bigl(\partial U\times[0,t_c)\bigr).
$$
This set is obtained from $U$ by attaching a thin outward collar. Therefore it is diffeomorphic to $U$, and hence is simply connected. 
\end{proof}

\begin{remark}
The perturbation constructed in the proof of \autoref{thm_gredient_growth} can actually be made to be arbitrarily small in $C^k$ sense for any $k\ge 0$, provided that $\omega_0 \in C^k$ by adjusting the amplitude $a$ of the weak part. However, it is still out of our reach to construct a perturbation $\tilde{\omega}_0$ near a given $\omega_0$ such that the support of $\tilde{\omega}_0 - \omega_0$ has small measure, which means our perturbation is not a local one. 
For a background datum $\omega_0$ that is close to a point vortex, the area we require for the support could be very large.

Indeed a local perturbation may lead to stability as suggested in \cite{Choi2024c}, where a stability result is given for nearly point-vortex initial datum with some additional symmetry assumptions on configuration. 
See for example \cite{Marchioro1994, Zbarsky2021, Meyer2025a} and the references therein for more discussions on point vortex configuration problem, and \cite{Butta2025, Donati2025a} for similar problem in the context of 3D axisymmetric Euler system.
\end{remark}

\section{Extensions}

\subsection{The Quadrant Case}

For the special case of odd-odd symmetry, that is, $$
\omega_0(x_1,x_2)=-\omega_0(-x_1,x_2)=-\omega_0(x_1,-x_2)=\omega_0(-x_1,-x_2),
$$
from uniqueness of Yudovich solutions, $\omega(t,x)$ preserves the odd-odd symmetry for all time. Denote $$
Q:=\{x=(x_1, x_2)\in \R^2: x_1 > 0, x_2 > 0\},
$$
then any odd-odd solution can be considered as the solution in the first quadrant $Q$ with slip boundary condition, i.e., $v\cdot n = 0$ on $\p Q$.
The corresponding Biot-Savart kernel then becomes 
\begin{equation}
\label{kernel_quadrant}
K(x,y)=\frac{1}{2\pi}\left(\frac{(x-y)^\perp}{|x-y|^2}-\frac{(x-\bar{y})^\perp}{|x-\bar{y}|^2}-\frac{(x+\bar{y})^\perp}{|x+\bar{y}|^2}+\frac{(x+y)^\perp}{|x+y|^2}\right).
\end{equation}
The parallel estimate of mass center in the quadrant can be found in \cite{Iftimie1999}, which is given by
$$
P_1'(t) \ge 10^{-3} \frac{E^{6}}{m^{10} M P_2(0)^3}, P_2'(t) \le 0.
$$
This implies that the mass center moves at least linearly, which leads to the linear growth of the diameter of support by the same argument as in the proof of \autoref{thm_diameter_growth_quadrant}, and hence yields the parallel results on the growth of the vorticity gradient (\autoref{thm_gredient_growth} and \autoref{thm_gredient_integral_growth}). However, the construction can be much simpler in the quadrant case since we can let the support touch the $x_2-$axis, and the symmetry nature of odd-odd solutions keeps the support in the same axis, which directly yields the linear growth.

\subsection{gSQG System in the Half-Plane}

The generalized surface quasi-geostrophic (gSQG) system is another important 2D active scalar model, which is given for each $0\le \alpha\le 1 / 2$ by

\begin{equation}\label{gsqg}\tag{gSQG}
\left\{
\begin{aligned}
\partial_t \theta+v\cdot \nabla \theta = 0 &\text{ in } \R^2,\\
v= \nabla^\perp (-\Delta)^{-1+\alpha} \theta &\text{ in } \R^2,\\
\theta|_{t=0} =\theta_0. &
\end{aligned}
\right.
\end{equation}

The case $\alpha=0$ corresponds to the 2D Euler system, while $\alpha=1 / 2$ corresponds to the surface quasi-geostrophic (SQG) system. We prove the analogue of our main result for the gSQG system in the half-plane $\R^2_+$ as $0 < \alpha < 1 / 2$. In this setting, the Biot-Savart law is given by
\begin{equation}
K_\alpha(x,y) = C_\alpha \left(\frac{(x-y)^\perp}{|x-y|^{2+2\alpha}} - \frac{(x-\bar{y})^\perp}{|x-\bar{y}|^{2+2\alpha}}\right).
\end{equation}

\begin{theorem}\label{thm_gsqg_growth}
Consider the gSQG system \eqref{gsqg} for $0 < \alpha < \frac{1}{2}$ in the half-plane $\R^2_+$ with non-zero initial data $\theta_0 \in L^1 \cap L^\infty(\R^2_+)$, $\theta_0 \ge 0$ and $\supp \theta_0$ is bounded. 
Then, for each $\eps>0$, there exists a smooth perturbation $\tilde{\theta}_0$ of $\theta_0$ such that $\|\tilde{\theta}_0 - \theta_0\|_{L^1} < \eps$, and
the corresponding solution $\tilde{\theta}(t)$ to \eqref{gsqg} with initial data $\tilde{\theta}_0$ satisfies that
$\|\nabla \tilde{\theta}(t)\|_{L^\infty}$ grows at least linearly in $t\in [0,T)$, where $T$ is the lifespan of the solution. 
\end{theorem}

The proof is completely parallel to the Euler case, once we obtain the drift estimate for the horizontal center of mass as in the Euler case \eqref{VP_half_plane}.

\begin{lemma}\label{lem_drift_gsqg}
For the gSQG system \eqref{gsqg} in the half-plane $\R^2_+$, we have the following drift estimate for the horizontal center of mass: For $\theta_0 \in L^1 \cap L^\infty(\R^2_+)$ with $\theta_0 \ge 0$ and $\supp \theta_0 \subset \R^2_+$ is bounded, a constant $C=C(\alpha,m,M,P_2(0),E)$ exists such that $P_1'(t) \ge C$ for all $t\in[0,T)$, where $T$ is the lifespan of the solution.
\end{lemma}
\begin{proof}
The proof is based on the conservation of energy, that
\begin{equation}\label{eq_energy_gsqg}
\begin{aligned}
E &= \int_{\R^2_+}\int_{\R^2_+} \left[\frac{1}{|x-y|^{2\alpha}} - \frac{1}{|x-\bar{y}|^{2\alpha}} \right] \theta(t,x)\theta(t,y)\,\mathrm{d}x\, \mathrm{d}y \\
\end{aligned}
\end{equation}
remains constant for all time. 
Differentiating the first moment and using the Biot--Savart law gives
$$
P_1(t) = \frac{1}{m} \int_{\R^2_+} x_1 \theta(t,x)\,\d x ,
\qquad
P_1'(t) = \int_{\R^2_+}u_1(x)\theta(x)\,\d x  .
$$
Since
$$
    u_1(x)
    =
    c_\alpha
    \int_{\R^2_+}
    \left(
        -\frac{x_2-y_2}{|x-y|^{2\alpha+2}}
        +
        \frac{x_2+y_2}{|x-\bar y|^{2\alpha+2}}
    \right)\theta(y)\,\d y,
$$
the first term is odd under the exchange $x\leftrightarrow y$. Hence
$$
    P_1'(t)
    =
    c_\alpha
    \iint_{\R^2_+\times\R^2_+}
    \frac{x_2+y_2}{|x-\bar y|^{2\alpha+2}}
    \theta(x)\theta(y)\,\d x \d y .
$$
Similarly, $P_2'(t)=0$, and therefore $P_2$ is conserved.
By H\"{o}lder's inequality,
$$
    E\le C_\alpha P_1'(t)^{1/q} \cdot
    \left(
        \iint_{\R^2_+\times\R^2_+}
        \left[
            \frac{
                \left(
                    |x-y|^{-2\alpha}
                    -
                    |x-\bar y|^{-2\alpha}
                \right)^q
            }{
                (x_2+y_2)|x-\bar y|^{-(2\alpha+2)}
            }
        \right]^{1/(q-1)}
        \theta(x)\theta(y)\,\d x \d y
    \right)^{(q-1)/q}.
$$
The proof can be concluded by bounding the last term on the right-hand side. Since
$$
    |x-y|^{-2\alpha}-|x-\bar y|^{-2\alpha}
    =
    |x-y|^{-2\alpha}
    \left[
        1-
        \left(
            1+\frac{4x_2y_2}{|x-y|^2}
        \right)^{-\alpha}
    \right] \le
    C_\alpha
    \frac{(x_2y_2)^{1/2}}
    {|x-y|^{2\alpha} |x-\bar y|^{2\alpha}},
$$
and $x_2y_2\le (x_2+y_2)^2/4$, taking $q=2\alpha+2$ yields
$$
\begin{aligned}
\left[\frac{\left(|x-y|^{-2\alpha}-|x-\bar y|^{-2\alpha}\right)^q}
{(x_2+y_2)|x-\bar y|^{-(2\alpha+2)}}\right]^{1/(q-1)}
&\le C_\alpha \frac{x_2+y_2}{|x-y|^{\alpha q/(q-1)}|x-\bar{y}|^{(q-\alpha-2)/(q-1)}}
\\
& = C_\alpha (x_2+y_2)|x-y|^{-\beta},
\end{aligned}
$$
where
$$
    \beta=\frac{2\alpha q}{q-1}
    =
    \frac{4\alpha(\alpha+1)}{2\alpha+1}.
$$
Since $0<\alpha<1 / 2$, we have $0<\beta<2$. The following $L^\infty$ estimate similar to \autoref{lem_velocity_estimate} can be found in \cite[Lemma 2.2.2]{Iftimie2007}:
$$
    |\int_{\R^2_+}
    \frac{\theta(y)}{|x-y|^\beta}\,\d y|
    \le
    C_\beta m^{1-\beta/2}M^{\beta/2}.
$$
Consequently,
$$
\begin{aligned}
    &\iint_{\R^2_+\times\R^2_+}
    (x_2+y_2)|x-y|^{-\beta}
    \theta(x)\theta(y)\,\d x \d y = 2\int_{\R^2_+} x_2 \theta(x) \int_{\R^2_+} \frac{\theta(y)}{|x-y|^\beta}\,\d y \,\d x
    \\
    &\qquad\le
    C_\alpha
    P_2(t)\,m^{1-\beta/2}M^{\beta/2}.
\end{aligned}
$$
Combining the previous estimates gives
$$
    E_\alpha
    \le
    C_\alpha
    P_1'(t)^{1/q}
    \left(
        P_2(t)\,m^{1-\beta/2}M^{\beta/2}
    \right)^{(q-1)/q}.
$$
Raising to the power $q$ and using the formula for $P_1'(t)$, we obtain
$$
    P_1'(t)
    \ge
    c_\alpha
    \frac{
        E_\alpha^q
    }{
        \left(
            P_2(t)\,m^{1-\beta/2}M^{\beta/2}
        \right)^{q-1}
    } .
$$
The quantities $m,M,P_2(t)=P_2(0)$, and $E_\alpha$ are conserved. Hence the right-hand side
is a positive constant depending only on the initial data.
\end{proof}

\bibliographystyle{amsplain}
\bibliography{main}

\end{document}